\documentclass{amsart}
\usepackage{amsfonts}
\usepackage{amsmath,amssymb}
\usepackage{amsthm}
\usepackage{amscd}
\usepackage{graphics}
\usepackage{graphicx}

\theoremstyle{remark}{
\newtheorem{Def}{{\rm Definition}}
\newtheorem{Ex}{{\rm Example}}
\newtheorem{Rem}{{\rm Remark}}

}
\theoremstyle{plain}{

\newtheorem{Prop}{Proposition}
\newtheorem{Thm}{Theorem}
\newtheorem{MainThm}{Main Theorem}
\newtheorem{Lem}{Lemma}

}

\begin{document}
\title[Compositions of fold maps having simple structures with projections]{On simple classes of special generic maps and round fold maps and fold maps obtained by composing projections}
\author{Naoki Kitazawa}
\address{Institute of Mathematics for Industry, Kyushu University, 744 Motooka, Nishi-ku Fukuoka 819-0395, Japan\\
 TEL (Office): +81-92-802-4402 \\
 FAX (Office): +81-92-802-4405 \\
}
\email{n-kitazawa@imi.kyushu-u.ac.jp}
\urladdr{https://naokikitazawa.github.io/NaokiKitazawa.html}
\maketitle
\begin{abstract}

{\it Fold} maps are fundamental tools in the theory of singularities of differentiable maps and its applications to geometry. They are higher dimensional variants of Morse functions. Classes of {\it special generic} maps and {\it round} fold maps are important classes of fold maps. {\it Special generic} maps are higher dimensional variants of Morse functions on homotopy spheres with exactly two {\it singular points}: canonical projections of unit spheres are special generic. {\it Round} fold maps are Morse functions obtained as doubles of Morse functions, or fold maps such that the set of all the singular points are embeddings and that the images are concentric. In the present paper, we discuss compositions of these maps with canonical projections. For example, we observe that these compositions for special generic maps of simple classes are regarded as round fold maps in considerable cases. We also present round fold maps we cannot represent in this way, seeming to be represented so. Note that such compositions are natural operations in related theory of differentiable maps.

\end{abstract}


\maketitle
\section{Introduction.}
\label{sec:1}
\subsection{Fundamental notions and notation on differentiable maps and differential topology of manifolds and fold maps.}
{\it Fold} maps are fundamental tools in the theory of singularities of differentiable maps and its applications to geometry.
A {\it singular} point $p \in X$ of a differentiable map $c:X \rightarrow Y$ is a point at which the rank of the differential ${dc}_p$ is smaller than  both the dimensions $\dim X$ and $\dim Y$. $S(c)$ denotes the set of all the singular points of $c$ (the {\it singular set} of $c$). $c(S(c))$ is the {\it singular value set} of $c$. $Y-c(S(c))$ is the {\it regular value set} of $c$. 
A point in $Y$ is a {\it singular {\rm (}regular{\rm )} value} if it is a point in the singular (resp. regular) value set of the map.
\begin{Def}
\label{def:1}
A {\it fold} map is a smooth map at each singular point of which we can represent as $(x_1, \cdots, x_m) \mapsto (x_1,\cdots,x_{n-1},{\sum}_{j=n}^{m-i(p)}{x_j}^2-{\sum}_{j=m-i(p)+1}^m {x_j}^2$
for suitable coordinates and an integer $0 \leq i(p) \leq \frac{m-n+1}{2}$.
\end{Def}
We call a singular point of a smooth map represented as in Definition \ref{def:1} for suitable coordinates near the singular point and the value at the point a {\it fold} point for the map.
\begin{Prop}
In Definition \ref{def:1} and for a fold point $p$ for a general smooth map, we can determine $i(p)$ uniquely for each singular point $p$. The restriction to the set consisting of all the singular points of a fixed index is an immersion of an {\rm (}$n-1${\rm )}-dimensional closed and smooth submanifold with no boundary in the $m$-dimensional manifold of the domain.
\end{Prop}

We call $i(p)$ here the {\it index} of the singular point $p$. Note that a fold map is a Morse function if and only if $n=1$. 

{\it Special generic} maps and {\it round} fold maps are important classes of fold maps.
\begin{Def} 
\label{def:1}
A {\it special generic} map is a fold map such that $i(p)=0$ for any singular point $p$.
\end{Def}
As simplest examples, Morse functions with exactly two singular points on homotopy spheres in the so-called Reeb's theorem, canonical projections of unit spheres, and so on, are special generic. 

Hereafter, (boundary) connected sums of manifolds, are discussed in the smooth category. 
${\mathbb{R}}^k$ denotes the $k$-dimensional Euclidean space, endowed with the Euclidean metric, and for $x \in {\mathbb{R}}^k$, $||x|| \geq 0$ denotes the distance between $x$ and the origin $0 \in {\mathbb{R}}^k$ or the norm where $x$ is seen as a vector. $D^k:=\{x \in {\mathbb{R}}^{k} \mid ||x|| \leq 1.\}$ denotes the $k$-dimensional unit disc for $k>0$ and $S^k:=\{x \in {\mathbb{R}}^{k+1} \mid ||x||=1.\}$ denotes the $k$-dimensional unit sphere.  
\begin{Ex}
Let $m\geq n$ be positive integers. An $m$-dimensional closed manifold $M$ represented as a connected sum of $l>0$ manifolds $S^{l_j} \times S^{m-l_j}$ ($1 \leq l_j \leq n-1$) admits a special generic map $f:M \rightarrow {\mathbb{R}}^n$ such that $f {\mid}_{S(f)}$ is an embedding and that $f(M)$ is represented as a boundary connected sum of the $l$ manifolds $S^{l_j} \times D^{n-l_j}$ where $1 \leq j \leq l$ is an integer. 
\end{Ex}

Diffeomorphisms on a smooth manifold are assumed to be smooth. We define the {\it diffeomorphism group} of the manifold as the group of all the diffeomorphisms. The structure groups of bundles whose fibers are (smooth) manifolds are subgroups of the diffeomorphism groups unless otherwise stated. In other words the bundles are {\it smooth}. 
The class of {\it linear} bundles is a subclass of the class of smooth bundles. A bundle is {\it linear} if the fiber is regarded as a unit sphere or a unit disc in a Euclidean space and the structure group acts linearly in a canonical way.

\begin{Prop}[\cite{saeki}, \cite{saeki2}, and so on.]
\label{prop:2}
Let $m>n \geq 1$ be integers. 
An $m$-dimensional closed, connected and smooth manifold $M$ admits a special generic map into ${\mathbb{R}}^n$ if and only if the following three hold. 
\begin{enumerate}
\item There exists a smooth surjection $q_f:M \rightarrow W_f$ onto an $n$-dimensional compact and smooth manifold $W_f$ and an immersion $\bar{f}:W_f \rightarrow {\mathbb{R}}^n$. 
\item There exists a small collar neighborhood $C(\partial W_f)$ such that the composition of $q_f {\mid}_{{q_f}^{-1}(C(\partial W_f))}$ with the canonical projection to $\partial W_f$ gives a trivial linear bundle whose fiber is diffeomorphic to $D^{m-n+1}$.
\item $q_f {\mid}_{{q_f}^{-1}(W_f-{\rm Int}(C(\partial W_f)))}$ gives a smooth bundle whose fiber is diffeomorphic to $S^{m-n}$.
\end{enumerate}
Furthermore, we can take a special generic map $f$ as $\bar{f} \circ q_f$ satisfying $q_f(S(f))=\partial W_f$ if $M$ admits a special generic map into ${\mathbb{R}}^n$.
\end{Prop}

In the present paper, we mainly consider a simple subclass of special generic maps. 

We review {\it round} fold maps.
\begin{Def}[\cite{kitazawa0.1}, \cite{kitazawa0.2}, \cite{kitazawa0.3}, \cite{kitazawa0.10}, and so on,]
\label{def:3}
A fold map $f$ on an $m$-dimensional closed, connected and smooth manifold $M$ into ${\mathbb{R}}^n$ is said to be {\it round} if either of the following hold.
\begin{enumerate}
\item $n=1$. $f {\mid}_{S(f)}$ is an embedding. There exist a regular value $a \in {\mathbb{R}}^n$ and
 a pair $(\Phi:f^{-1}((-\infty,a]) \rightarrow f^{-1}([a,\infty)),{\phi}_{\Phi}: (-\infty,a] \rightarrow [a,\infty))$ such that $f {\mid}_{f^{-1}([a,\infty))} \circ \Phi={\phi}_{\Phi} \circ f {\mid}_{f^{-1}((-\infty,a])}$.
\item $n=2$. $f {\mid}_{S(f)}$ is an embedding. There exist a diffeomorphism $\phi$ on ${\mathbb{R}}^n$ and an integer $l>0$ such that $\phi(f(S(f)))=\{||x||=r \mid r \in \mathbb{N}, 1 \leq r \leq l\}$.  
\end{enumerate} 
\end{Def}
\begin{Def}[In \cite{kitazawa0.4} these notions are defined first essentially and we change the names of these notions]
\label{def:4}
Let $f:M \rightarrow {\mathbb{R}}^n$ be a round fold map and $\phi$ be the diffeomorphism in Definition \ref{def:3}.
\begin{enumerate}
\item
\label{def:4.1}
$f$ is said to {\it have a globally trivial monodromy} if either of the following holds.
\begin{enumerate}
\item $n=1$.
\item $n \geq 2$ and for a diffeomorphism $\phi$ on ${\mathbb{R}}^n$ and an integer $l>0$ such that $\phi(f(S(f)))=\{||x||=r \mid r \in \mathbb{N}, 1 \leq r \leq l\}$, the composition of the restriction of $\phi \circ f$ to ${(\phi \circ f)}^{-1}(\{||x||=r \mid \frac{1}{2} \leq r\})$ with a canonical map mapping $x$ to $\frac{1}{||x||} x$ gives a trivial bundle over the unit sphere.
\end{enumerate}
\item
\label{def:4.2}
$f$ is said to {\it have a componentwisely trivial monodromy} if either of the following holds. 
\begin{enumerate}
\item $n=1$. 
\item $n \geq 2$, and for a diffeomorphism $\phi$ on ${\mathbb{R}}^n$ and an integer $l>0$ such that $\phi(f(S(f)))=\{||x||=r \mid r \in \mathbb{N}, 1 \leq r \leq l\}$, the composition of the restriction of $\phi \circ f$ to ${(\phi \circ f)}^{-1}(\{||x||=r \mid k-\frac{1}{2} \leq r \leq k+\frac{1}{2} \})$ with a canonical map mapping $x$ to $\frac{1}{||x||} x$ gives a trivial bundle over the unit sphere for each integer $1 \leq k \leq l$. 
\end{enumerate}
\end{enumerate}
\end{Def}
As a simplest example, canonical projections of unit spheres are round and have globally trivial monodromies and componentwisely trivial monodromies.
\subsection{Main Theorems.}
The following three are part of main theorems of the present paper. They are all on compositions of special generic maps or round fold maps satisfying good properties with suitable projections. Composing smooth maps with projections are natural and in various scenes strong methods in the singularity theory and geometric theory of differentiable maps. \cite{mather} is one of pioneering papers on sophisticated methods of the singularity theory of differentiable maps and concentrates on generic properties on singular points or singularities of smooth maps obtained by composing given smooth maps with canonical projections. This is to some extent related to the present study where we do not need to understand the theory so much here. \cite{fukuda} considers compositions of smooth maps of a good class or the class of so-called {\it Morin} maps with canonical projections to the line or the $1$-dimensional Euclidean space to relate topological information of singular sets with the manifolds of the domains via classical and fundamental theory of Morse functions. The class of Morin maps contains the class of fold maps as a simple class.
We do not concentrate on this here. 
\cite{golubitskyguillemin} explains fundamental and sophisticated theory on the singularity theory and geometric theory of differentiable maps systematically including fundamental theory of Morse functions, fold maps and Morin maps, well-known Mather's sophisticated theory including theory related to \cite{mather}, and so on.

${\pi}_{m,n}:{\mathbb{R}}^m \rightarrow {\mathbb{R}}^n$ denotes the canonical projection to the first $n$ components where $m>n \geq 1$.
\begin{MainThm}
\label{mainthm:1}
Let $m>n \geq 2$ and $l>0$ be integers. Let $M$ be an $m$-dimensional closed manifold and $f: M \rightarrow {\mathbb{R}}^n$ be a round fold map satisfying the following three properties.
\begin{enumerate}
\item The index of each singular point is $0$ or $1$. The number of connected components of the singular set is $l+1$.
\item Preimages of regular values of $f$ are disjoint unions of copies of $S^{m-n}$ and the numbers of the connected components of preimages of regular values in the connected component containing $0$ are $l+1$ where $\phi$ denotes the diffeomorphism in {\rm Definition \ref{def:3}}. 
\end{enumerate}
Then by composing $\phi \circ f$ with the canonical projection ${\pi}_{n,n^{\prime}}:{\mathbb{R}}^n \rightarrow {\mathbb{R}}^{n^{\prime}}$, we have a new round fold map satisfying the following two properties.
\begin{enumerate}
\item The index of each singular point is $0$ or $\min\{n-n^{\prime}+1,\frac{m-n^{\prime}+1}{2},(m-n^{\prime}+1)-(n-n^{\prime}+1)\}$. The number of singular points is $2(l+1)$ for $n^{\prime}=1$ and that of connected components of the singular set is $l+1$ for $n^{\prime} \geq 2$.
\item Preimages of regular values are diffeomorphic to $S^{m-n^{\prime}}$ or represented as connected sums of copies of $S^{n-n^{\prime}} \times S^{m-n}$ and the numbers of the copies of $S^{n-n^{\prime}} \times S^{m-n}$ of preimages of regular values in the connected component containing $0$ are $l+1$.
\item The resulting round fold map has a globally trivial monodromy.
\end{enumerate}
\end{MainThm}
In the following two Main Theorems, we need the notion of a fold map {\it represented as a connected sum of fold maps} and a {\it summand} for the family of the fold maps, for example. We define these notions in Definition \ref{def:5}.
\begin{MainThm}
\label{mainthm:2}
In Proposition \ref{prop:2}, let $M$ admit a special generic map $f:M \rightarrow {\mathbb{R}}^n$ such that $n \geq 2$, and represented as a connected sum of a family of finitely many fold maps indexed by $j \in J$ satisfying the following three.
\begin{enumerate}
\item The images of the special generic maps regarded as summands for the family of the maps are smoothly immersed manifolds diffeomorphic to $F_j \times [-1,1]$ where $F_j$ is an {\rm (}$n-1${\rm )}-dimensional closed and connected manifold we can
immerse into ${\mathbb{R}}^n$. 
\item $W_f$ in Proposition \ref{prop:2} is represented as a boundary connected sum of these manifolds $F_j \times [-1,1]$.
\item For each manifold $F_j$, we can take an immersion before so that the composition of the immersion of $F_j$ with the canonical projection to ${\mathbb{R}}^{n_j}$ via ${\pi}_{n,n_j}$ is a round fold map into ${\mathbb{R}}^{n_j}$.
\end{enumerate}
Then by changing $\bar{f}$ to a suitable immersion and composing the new special generic map with ${\pi}_{n,\min \{n_j \mid j \in J.\}}$, we have a round map into ${\mathbb{R}}^{\min \{n_j \mid j \in J.\}}$.
\end{MainThm}
The last one is a corollary to existing and obtained results. Theorem \ref{thm:2} and Lemma \ref{lem:1} are in the next section and the former is also a main theorem. We do not introduce these two rigorously in the present section.
\begin{MainThm}
\label{mainthm:3}
There exists a family of infinitely many $7$-dimensional closed, simply-connected and spin manifolds satisfying the following five.
\begin{enumerate}
\item The integral cohomology rings of these 7-dimensional manifolds are mutually isomorphic. 
\item Distinct 7-dimensional manifolds in the family are not homeomorphic.
\item All manifolds in the family admit round fold maps having globally trivial monodromies and satisfying the property on preimages of regular values for round maps obtained in Theorem \ref{thm:2} with $(m,n,n^{\prime})=(7,6,4)$ and $(i_1,i_2)=(1,1)$.
\item Previous round fold maps cannot be obtained in Theorem \ref{thm:2} with $(m,n,n^{\prime})=(7,6,4)$.
\item Each of these round fold maps cannot be obtained by composing any fold map into ${\mathbb{R}}^5$ satisfying the following properties with the canonical projection. 
\begin{enumerate}
\item A fold map is represented as a connected sum of finitely many fold maps.
\item Each summand for the family of the previous fold maps is a round fold map into ${\mathbb{R}}^5$ as in the assumption of Main Theorem \ref{mainthm:1} or a special generic map in Lemma \ref{lem:1}.
\end{enumerate}
\end{enumerate}
\end{MainThm}
\subsection{The content of the present paper.}
In the second section, we prove Main Theorem \ref{mainthm:1}. We also see that a main ingredient of the proof generalize a main argument of section 6 of \cite{saekisuzuoka}. We also introduce Theorem \ref{thm:1} as a previously obtained and related result by the author. As another main theorem, we prove Theorem \ref{thm:2}.
In the third section, we prove Main Theorem \ref{mainthm:2}. Some of \cite{kitazawa0.4}--\cite{kitazawa0.6} are closely related to this.
In the fourth section, we prove Main Theorem \ref{mainthm:3}. Related to this, we first give a further exposition on Theorem \ref{thm:1} as Theorem \ref{thm:3}. As Remark \ref{rem:3}, we also explain relations between the arguments and ones in \cite{kitazawa0.7}--\cite{kitazawa0.9}, which are on construction of new fold maps such that the compositions with suitable projections are given fold maps. In other words, we consider lifting given fold maps to new fold maps. This construction is an important and explicit work in considering lifting smooth maps to immersions, embeddings or more general generic smooth maps. 

Hereafter, manifolds, maps between manifolds, and so on, are smooth unless otherwise stated.
\section{A proof of Main Theorem \ref{mainthm:1} and related problems, results and remarks.}

\begin{proof}[A proof of Main Theorem \ref{mainthm:1}]
The fact that we have a new round fold map is obvious from the definition of a fold map and fundamental theory. 
The fact on the indices of singular points are obvious from the local forms of the round fold maps.
We consider the preimage of a regular value $p \in {\mathbb{R}}^{n^{\prime}}$ such that the distance between the origin and this is $k<||p||<k+1$ for a non-negative integer $k\leq l$.

Consider the canonical inclusion of ${\mathbb{R}}^{n^{\prime}}$ into ${{\mathbb{R}}^{n^{\prime}}} \times \{0\} \subset {{\mathbb{R}}^{n^{\prime}}} \times {{\mathbb{R}}^{n-n^{\prime}}} = {\mathbb{R}}^n$.
The preimage is regarded as the preimage of $\{p\} \times {{\mathbb{R}}^{n-n^{\prime}}}$ for the original round fold map.
The preimage is also regarded as the boundary of the preimage of the set $\{tp\mid t \geq 1\} \times {\mathbb{R}}^{n-n^{\prime}}$ for the original round fold map: $tp$ is represented as a vector and regarded as a point canonically.

It is diffeomorphic to the boundary of the product of the following two manifolds where the corner is eliminated. 

\begin{enumerate}
\item The preimage of $\{tp\mid t \geq 1\} \times \{0\} \subset \{tp\mid t \geq 1\} \times {\mathbb{R}}^{n-n^{\prime}}$, which is diffeomorphic to a manifold obtained by removing the interiors of $l+1-k$ disjointly and smoothly embedded ($m-n+1$)-dimensional unit discs from a copy of $S^{m-n+1}$.
\item $\{tp\mid t \geq 1\} \times {\mathbb{R}}^{n-n^{\prime}} \bigcap \partial (\phi \circ f)(M)$, which is diffeomorphic to $D^{n-n^{\prime}}$.  
\end{enumerate}

FIGURE \ref{fig:1} shows a case where $(n,l)=(2,1)$ and $p \in \mathbb{R}$ with $0<||p||<1$.

\begin{figure}
\includegraphics[width=40mm]{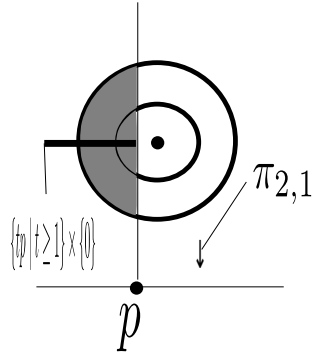}
\caption{The image (the singular value) of $\phi \circ f$ for $(n,k)=(2,1)$ and $p \in \mathbb{R}$ with $0<||p||<1$.}
\label{fig:1}
\end{figure}

From the structure of the original round fold map, the bundle in Definition \ref{def:4} (\ref{def:4.1}) is regarded as the restriction of the bundle over the unit sphere $S^{n-1}$ to the equator and also regarded as the restriction of the bundle over an ($n-1$)-dimensional hemisphere (in the unit sphere) to the boundary. This completes the proof of the third statement.
This completes the proof.

\end{proof}
The following present some examples of given round fold maps in Main Theorem \ref{mainthm:1}.
\begin{Thm}[\cite{kitazawa0.1}, \cite{kitazawa0.2} and \cite{kitazawa0.10}.]
\label{thm:1}
Let $m>n \geq 1$ be integers. 
Let $M$ be an $m$-dimensional manifold represented as a connected sum of $l>0$ total spaces of bundles over $S^n$ whose fibers are diffeomorphic to $S^{m-n}$. Then
there exists a round fold map $f: M \rightarrow {\mathbb{R}}^n$ having a componentwisely trivial monodromy satisfying the following two properties.
\begin{enumerate}
\item The index of each singular point is $0$ or $1$. The number of singular points is $2(l+1)$ for $n=1$ and that of connected components of the singular set is $l+1$ for $n \geq 2$.
\item Preimages of regular values of $f$ for $n=1$ or $\phi \circ f$ for $n \geq 2$ are disjoint unions of copies of $S^{m-n}$ and the numbers of the connected components of preimages of regular values in the connected component containing $p$ are $l+1$ where $\phi$ denotes the diffeomorphism in Definition \ref{def:3}. Furthermore, $p=a$ if $n=1$ and $p=0$ if $n \geq 2$.
\end{enumerate}
\end{Thm}

\begin{Rem}
\label{rem:1}
Related to the argument here, Asano \cite{asano} considered and solved a more general problem for a good class of smooth maps into the plane on $4$-dimensional closed, connected and orientable manifolds. This study concentrates on the class of so-called {\it trisection maps} into the plane. The image of such a map is a smoothly embedded copy of the $2$-dimensional unit disc and the singular value set is, topologically, embedded concentric circles and may have finitely many {\it cusps} where the singular sets are $1$-dimensional closed and smooth submanifolds with no boundaries and where all singular points except (preimages of) cusps just before are fold points. One of the problems questions, for a so-called {\it generic} arc properly embedded in the image, to which $3$-dimensional closed manifold the preimage is diffeomorphic?
For general theory of the class of trisection maps, see \cite{gaykirby} and as a recent work, see also \cite{baykursaeki} for example. 
\end{Rem}
\begin{Def}
\label{def:5}
A fold map $f$ on an $m$-dimensional closed and connected manifold $M$ into ${\mathbb{R}}^n$ is said to be {\it represented as a connected sum} of two fold maps if there exist a hyperplane $H={\mathbb{R}}^{n-1} \times \{0\} \subset {\mathbb{R}}^n$ with its tubular neighborhood ${\mathbb{R}}^{n-1} \times [-1,1]=H \times [-1,1]$ and a diffeomorphism $\phi$ from a submanifold $S \times [-1,1] \subset {\mathbb{R}}^n$ regarded as the total space of a product bundle over $S$, diffeomorphic to ${\mathbb{R}}^{n-1}$, onto $H \times [-1,1]$, satisfying the following three where $f_S$ denotes the restriction of $f$ to $f^{-1}(S \times [-1,1])$.
\begin{enumerate}
\item The restriction of $\phi$ to $S \times \{t\}$ is a diffeomorphism onto $H \times \{t\}$ for any $t \in [-1,1]$.  
\item There exists a diffeomorphism ${\Phi}_H:S^{m-1} \times [-1,1] \rightarrow f^{-1}(S \times [-1,1])$ such that the restriction of ${\Phi}_H$ to $S^{m-1} \times \{t\}$ is a diffeomorphism onto $f^{-1}(S \times \{t\})$.
\item ${\phi} \circ f_S \circ {\Phi}_H$ is the canonical projection of the ($m-1$)-dimensional unit sphere into the ($n-1$)-dimensional Euclidean space if we restrict this to each $S^{m-1} \times \{t\}$, the target to each ${\mathbb{R}}^{n-1} \times \{t\}$, and identify them with $S^{m-1}$ and ${\mathbb{R}}^{n-1}$ in canonical ways.
\end{enumerate}
\end{Def}
By decomposing the fold map into two smooth maps and by attaching copies of the restriction of the canonical projection to the hemisphere which is the preimage of the subspace of the target regarded as the $n$-dimensional half-space in a suitable way, we have two fold maps such that the original manifold is represented as a connected sum of the two resulting manifolds. Note that the target of the projection of the hemisphere is the $n$-dimensional half-space. We call each of these fold maps a {\it summand} of $f$ for the pair of the fold maps. We can define these notions where we need to consider a family of fold maps consisting of more than $2$ fold maps and finitely many ones.

\begin{Lem}
\label{lem:1}
In Proposition \ref{prop:2}, let $M$ admit a special generic map $f:M \rightarrow {\mathbb{R}}^n$ such that $n \geq 2$, and that $W_f$ is represented as a boundary connected sum of a copy of $S^{n-1} \times I$ where $I$ denotes a closed interval.
Then by changing $\bar{f}$ to a suitable embedding and composing the new special generic map with ${\pi}_{n,n-1}$, we have a round fold map having a componentwisely trivial monodromy in the assumption of Main Theorem \ref{mainthm:1}.
\end{Lem}
\begin{proof}
By a diffeomorphism on ${\mathbb{R}}^n$, we can map $\bar{f}(W_f)$ to a submanifold obtained by removing the interiors of finitely many copies of the unit disc smoothly and disjointly embedded into the interior of the unit disc whose center is the origin. We take the embedded copies so that they are discs of fixed diameters, that their centers locate at points in an axis, and that for distinct copies, the diameters are distinct. By composing the resulting special generic map with a suitable projection, we have
a round fold map satisfying the conditions assumed in Main Theorem \ref{thm:1}. By arguments similar to one in the proof of the last property of the resulting round fold map in Main Theorem \ref{mainthm:1}, the resulting round fold map has a componentwisely trivial monodromy. To complete the proof, we may also apply the main ingredient of section 6 of \cite{saekisuzuoka}.
\end{proof}
This special generic map is regarded as one represented as a connected sum of finitely many special generic maps: for each summand $f_a$ for the family of fold maps, $W_{f_a}$ is diffeomorphic to $S^{n-1} \times I$. 

The following theorem is a newly obtained theorem in the present paper.
\begin{Thm}
\label{thm:2}
In Proposition \ref{prop:2}, let $m>n \geq 2$ and $M$ admit a special generic map $f:M \rightarrow {\mathbb{R}}^n$ such that 
$W_f$ is diffeomorphic to the $n$-dimensional unit disc or a manifold of the form $S^k \times D^{n-k}$ where $k$ is an integer satisfying $1 \leq k \leq n-1$, or a special generic map $f:M \rightarrow {\mathbb{R}}^n$ represented as a connected sum of a family of ${\Sigma}_{k=1}^{n-1} i_k>1$ fold maps satisfying the following two.
\begin{enumerate}
\item $i_k$ is a non-negative integer.
\item The images of exactly $i_k$ special generic maps regarded as summands for the family of the maps are smoothly immersed manifolds diffeomorphic to $S^{n-k} \times D^k$. $W_f$ in Proposition \ref{prop:2} is represented as a boundary connected sum of these manifolds.
\end{enumerate}
By changing $\bar{f}$ to a suitable embedding and composing the new special generic map with ${\pi}_{n,n^{\prime}}$, we have a round fold map having a globally trivial monodromy in the former case and a fold map satisfying the following two in the latter case.
\begin{enumerate}
\item The fold map is represented as a connected sum of finitely many fold maps and summands for the family of these fold maps satisfy the following two.
\begin{enumerate}
\item If $n^{\prime} \leq \max\{n-k \mid i_k>0\}$, then exactly one of the summands is a round fold map.
\item For the remaining summands, they are special generic and the images of exactly $i_k$ special generic maps are smoothly embedded manifolds diffeomorphic to $S^{n-k} \times D^{n^{\prime}-n+k}$ for $n-n^{\prime}+1 \leq k \leq n-1$. $W_f$ in Proposition \ref{prop:2} is represented as a boundary connected sum of these manifolds.
\end{enumerate}
\item Moreover, the round fold map before satisfies the following properties.
\begin{enumerate}
\item The singular value set can be of the form $\{x \in {\mathbb{R}}^{n^{\prime}} \mid ||x|| \in \mathbb{N}, 1 \leq ||x|| \leq {\Sigma}_{k=1}^{n-n^{\prime}} i_k+1.\}$ by composing the round fold map with a suitable diffeomorphism $\phi$ as in Definition \ref{def:3} if we need.
\item Consider going straight from a point in the complementary set of the image to the origin in the target of the round fold map. Then in the $j$-th connected component of the intersection of the image and the regular value set, the preimage of a point there is diffeomorphic to a manifold in the following list.
\begin{enumerate}
\item If $j=1$, then it is diffeomorphic to $S^{m-n^{\prime}}$.
\item If ${\Sigma}_{j^{\prime}=1}^{j_0} i_{j^{\prime}}+1 < j \leq {\Sigma}_{j^{\prime}=1}^{j_0+1} i_{j^{\prime}}+1$ where $0 \leq j_0< n-n^{\prime}-1$, then it is diffeomorphic to a manifold represented as a connected sum of the manifolds satisfying the following two.
\begin{enumerate}
\item The family of the manifolds consists of exactly $j-1$ manifolds.
\item The family contains exactly $i_{j^{\prime}}$ manifolds diffeomorphic to $S^{j^{\prime}} \times S^{m-n^{\prime}-j^{\prime}}$ for $1 \leq j^{\prime}<j_0$ and exactly $j-{\Sigma}_{j^{\prime}=1}^{j_0} i_{j^{\prime}}-1$
 manifolds diffeomorphic to $S^{j_0} \times S^{m-n^{\prime}-j_0}$.
\end{enumerate} 
\item ${\Sigma}_{j^{\prime}=1}^{n-n^{\prime}-1} i_{j^{\prime}}+1 < j \leq {\Sigma}_{j^{\prime}=1}^{n-n^{\prime}} i_{j^{\prime}}+1$, then it is diffeomorphic to the disjoint union of a manifold represented as a connected sum of the manifolds
 satisfying the following two and $j-{\Sigma}_{j^{\prime}=1}^{n-n^{\prime}-1} i_{j^{\prime}}-1$ copies of $S^{m-n^{\prime}}$.
\begin{enumerate}
\item The family of the manifolds consists of exactly ${\Sigma}_{j^{\prime}=1}^{n-n^{\prime}-1} i_{j^{\prime}}$ manifolds.
\item The family contains exactly $i_{j^{\prime}}$ manifolds diffeomorphic to $S^{j^{\prime}} \times S^{m-n^{\prime}-j^{\prime}}$ for $1 \leq j^{\prime} \leq n-n^{\prime}-1$.
\end{enumerate}
\end{enumerate}
\item The round fold map has a componentwisely trivial monodromy.
\end{enumerate}
\end{enumerate}
\end{Thm}
\begin{proof}
We can change the original special generic map $f$ by changing $\bar{f}$ into a suitable immersion so that the following properties fold. 

\begin{enumerate}
\item In a generalized case of Definition \ref{def:5}, we have finitely many hyperplanes $S$ each of which is of the form ${\mathbb{R}}^{n-1} \times \{a_i\}={\mathbb{R}}^n$ and we have a family of fold maps.
\item Summands for the family satisfy the following properties.
\begin{enumerate}
\item For one of the summand, in a generalized case of Definition \ref{def:5}, we have ${\Sigma}_{k=1}^{k=n-n^{\prime}} i_k$ hyperplanes each of which is of the form ${\mathbb{R}}^{n-1} \times \{b_i\}={\mathbb{R}}^n$ after composing the map with a suitable diffeomorphism on the target. Thus we have another family of fold maps. For the summands for the newly obtained family, there exist exactly $i_k$ special generic maps whose images are embedded manifolds diffeomorphic to $S^{n-k} \times D^k$ for $1 \leq k \leq n-n^{\prime}$.
\item For the remaining summands, there exist exactly $i_k$ special generic maps whose images are embedded manifolds diffeomorphic to $S^{n-k} \times D^k$ for $n-n^{\prime}+1 \leq k \leq n-1$. We also regard $S^{n-k} \times \{0\} \subset {\mathbb{R}}^{n-k} \times \{0\} \times I_k \subset {\mathbb{R}}^{n-k} \times {\mathbb{R}}^{n-n^{\prime}} \times I_k \subset {\mathbb{R}}^{n^{\prime}-1} \times {\mathbb{R}}^{n-n^{\prime}} \times I_k$ and $S^{n-k} \times D^k$ as a smoothly and naturally embedded compact manifold in the interior of ${\mathbb{R}}^{n-k} \times {\mathbb{R}}^{k-1} \times I_k={\mathbb{R}}^{n-1} \times I_k$ where $I_k$ denotes a closed interval in $\mathbb{R}$. 
\end{enumerate}
\end{enumerate}
After taking $\bar{f}$ suitably, we compose the resulting special generic map with the canonical projection to ${\mathbb{R}}^{n^{\prime}}$, identified canonically and suitably with ${\mathbb{R}}^{n^{\prime}-1} \times \{0\} \times I_k$. By applying Main Theorem \ref{mainthm:1} and Lemma \ref{lem:1} and respecting the structures of the special generic map and its image, we have a round map, which is also a desired map. This completes the proof.
\end{proof}
\begin{Rem}
\label{rem:2}
One of pioneering studies studying preimages of regular values for (so-called {\it generic}) maps obtained by composing given generic maps with projections is, presented in section 6 of \cite{saekisuzuoka}. This shows that preimages of regular values for (generic) maps obtained by composing special generic maps with projections are disjoint unions of copies of spheres. 
In Lemma \ref{lem:1}, specific cases of this explicit study appear. 
\end{Rem}

\section{A proof of Main Theorem \ref{mainthm:2} and related expositions.}
Hereafter, for a set $X$ $\sharp X$ denotes the cardinality of $X$. 
\begin{proof}[A proof of Main Theorem \ref{mainthm:2}.]
We can change the original special generic map $f$ by changing $\bar{f}$ into a suitable immersion so that the following properties fold. 

\begin{enumerate}
\item In a generalized case of Definition \ref{def:5}, we have $\sharp J-1$ hyperplanes $S$ each of which is of the form ${\mathbb{R}}^{n-1} \times \{a_i\}={\mathbb{R}}^n$ and we have a family of fold maps.
\item For the summands, there exist exactly $\sharp J$ special generic maps whose images are immersed manifolds diffeomorphic to $F_j \times I=F_j \times [-1,1]$ for each $j \in J$ and we can consider identifications between the domains of the immersions and $F_j \times I$ and make the situation satisfying the following two by virtue of the assumption on the immersion of $F_j$.
\begin{enumerate}
\item The compositions of the restrictions of the immersions to $F_j \times \{1\}$ and $F_j \times \{-1\}$ with the canonical projection to ${\mathbb{R}}^{n_j}$ round fold maps.
\item For the round fold maps just before, the disjoint union of singular value sets is of the form of the singular value set of a round fold map or in short concentric.  
\end{enumerate}
\end{enumerate}
After taking $\bar{f}$ suitably, we compose the resulting special generic map with the canonical projection to ${\mathbb{R}}^{\min \{n_j \mid j \in J.\}}$. 
Points which are not singular points of the special generic map are not singular for the resulting map. From the structure of the map and especially, that of the image of the special generic map, the resulting map is a desired round fold map.
\end{proof}
We discuss examples for Main Theorem \ref{mainthm:2}. For a manifold represented as a connected sum of finitely many copies of $S^{n^{\prime}} \times S^{n-n^{\prime}}$ where $n>n^{\prime} \geq 1$ be integers, we can construct an example of special generic maps into ${\mathbb{R}}^{n^{\prime}+1}$ satisfying the conditions in Example \ref{lem:1} and represented as a composition of an embedding into ${\mathbb{R}}^{n+1}$ with the canonical projection. First take a natural embedding into ${\mathbb{R}}^{n+1}$ and compose the embedding with the canonical projection. By applying Lemma \ref{lem:1}, we have a round fold map into ${\mathbb{R}}^{n^{\prime}}$. Thus we can take such manifolds as $F_j$ in Main Theorem \ref{mainthm:2}.

Related to this, with a little effort, we can find manifolds admitting special generic maps appearing in the summand here in \cite{kitazawa0.4}, \cite{kitazawa0.5}, \cite{kitazawa0.6}, and so on.
\section{A proof of Main Theorem \ref{mainthm:3} and related expositions.}
\begin{Thm}
\label{thm:3}
In Theorem \ref{thm:1}, let $(m,n)=(7,4)$ and for any $m$-dimensional homotopy sphere $M$, we have a round fold map into ${\mathbb{R}}^n$ satisfying the properties.
Furthermore, we can take $l=1$ if and only if $M$ is represented as the total space of a bundle over $S^4$ whose fiber is diffeomorphic to $S^3$ and $l \geq 2$ for any $m$-dimensional homotopy sphere.
\end{Thm}
$7$-dimensional oriented homotopy spheres of exactly $16$ of all the $28$ types of $7$-dimensional oriented homotopy spheres is represented as the total space of a bundle over $S^4$ whose fiber is diffeomorphic to $S^3$. The unit sphere satisfies this property.
According to \cite{calabi}, \cite{saeki}, \cite{saeki2}, and so on, a $7$-dimensional homotopy sphere admits a special generic map into ${\mathbb{R}}^n$ if and only if it is diffeomorphic to the unit sphere $S^7$ for $4 \leq n \leq 6$.
This means that maps in Theorem \ref{thm:2} cannot be obtained via theory of section 6 of \cite{saekisuzuoka} (if $M$ is not diffeomorphic to $S^7$).
\begin{proof}[A proof of Main Theorem \ref{mainthm:3}.]
Existence of the family of infinitely many manifolds satisfying the first three properties is due to \cite{kitazawa0.11}
 with \cite{wang}. The integral cohomology rings are isomorphic to that of ${\mathbb{C}P}^2 \times S^3$.
 If at least one of the remaining two properties is false, then the cohomology ring of a manifold in the family must be isomorphic to that of a manifold represented as a connected sum of $S^2 \times S^5$ and $S^3 \times S^4$.
 This completes the proof.
\end{proof}
\begin{Rem}
\label{rem:3}
\cite{kitazawa0.7}--\cite{kitazawa0.9} present fold maps such that preimages of regular values are disjoint unions of copies of a standard sphere we can represent as compositions  of special generic maps with canonical projections of Euclidean spaces. This is a new explicit topic on so-called lifts of smooth maps. 

In these studies, for a smooth map, we construct an explicit good smooth map such that the composition with a canonical projection of Euclidean space is the given map. Before the presented preprints by the author appeared, construction of explicit immersions and embeddings had been studied in various cases. \cite{haefliger} is one of pioneering studies. Later, related studies such as \cite{blankcurley}, \cite{levine}, \cite{nishioka}, \cite{saito}, \cite{saekitakase}, \cite{yamamoto}, and so on, have been published. 
\end{Rem}
\begin{Rem}
Special generic maps presented in the present paper and the class of such maps have been concentrated on and studied systematically in \cite{kitazawa0.12} for example.
\end{Rem}
\section{Acknowledgement.}
The author is a member of and supported by JSPS KAKENHI Grant Number JP17H06128 "Innovative research of geometric topology and singularities of differentiable mappings". The author declares that data concerning the present study directly are available within the present paper.

\end{document}